\documentclass{imsart}

\RequirePackage[OT1]{fontenc}
\RequirePackage{amsthm, amsmath, amsfonts, amssymb}
\RequirePackage[square]{natbib}
\RequirePackage[colorlinks,citecolor=blue,urlcolor=blue]{hyperref}
\RequirePackage{hypernat}

\startlocaldefs

\newcommand{\argmax}{\mathop{\rm{argmax}\mathstrut}}

\newcommand{\cl}{\mathop{\mathrm{cl}\mathstrut}\nolimits}
\newcommand{\ri}{\mathop{\mathrm{ri}\mathstrut}\nolimits}

\newcommand{\conv}{\mathop{\mathrm{conv}\mathstrut}\nolimits}

\newcommand{\supp}{\mathop{\mathrm{supp}\mathstrut}\nolimits}

\newcommand{\Real}{\mathbb{R}\mathstrut}

\newcommand{\RealP}{\mathbb{R}\mathstrut_{+}}

\theoremstyle{plain}
\newtheorem{Thm}{Theorem}[section]
\newtheorem{Lem}[Thm]{Lemma}

\theoremstyle{definition}

\theoremstyle{remark}

\endlocaldefs

\begin{document}

\begin{frontmatter}
\title{Uniqueness of the maximum likelihood estimator for $k$-monotone densities} 
\runtitle{Uniqueness of the $k$-monotone MLE}

\begin{aug}
    \author{\fnms{Arseni} \snm{Seregin}\thanksref{t1}\ead[label=e1]{arseni@stat.washington.edu}} 
	
	\thankstext{t1}{Research supported in part by NSF grant DMS-0804587}
	\runauthor{Arseni Seregin}


	\address{Department of Statistics, Box 354322\\University of Washington\\Seattle, WA  98195-4322\\
	\printead{e1}}
\end{aug}

\begin{abstract}
We prove uniqueness of the maximum likelihood estimator for the class of $k-$monotone densities.
\end{abstract}



\end{frontmatter}

\section{Introduction}
The family $\mathcal{M}_{k}$ of  $k$-monotone densities on $\RealP$ is a mixture model defined by the basic density $g_{0}(x)\equiv k(1-x)_{+}^{k-1}$. That means that every density $f\in\mathcal{M}_{k}$ is a multiplicative convolution of some mixing probability distribution $G$ and $g_{0}$:
$$
f(x) = \int_{0}^{+\infty} \frac{1}{y}\ g_{0}\left(\frac{x}{y}\right)dG(y).
$$
This class generalizes classes of monotone ($1-$monotone) and convex decreasing (2-monotone) densities on $\RealP$. As $k$ increases, the classes $\mathcal{M}_{k}$ decrease and the class of completely monotone densities lies in the intersection of all $\mathcal{M}_{k}$. \cite{BW-2009} proves consistency and \cite{MR2382657} gives asymptotic limit theory for the MLE of a $k-$monotone density. For $k=1$ the uniqueness of the MLE follows from a well known characterization of the Grenander estimator, for $k=2$ the uniqueness was proved in \cite{MR1891742} and for $k=3$ in \cite{Balabda-2004}. In \cite{MR653523} the uniqueness and weak consistency was proved for the MLE of a completely monotone density. In \cite{MR1234757} a general method was developed which allows to prove uniqueness of the MLE in mixture models defined by strictly totally positive kernels such as the exponential kernel for completely monotone densities. However, for $k-$monotone densities the kernel is not strictly totally positive as follows from \cite{MR0053177} and uniqueness of the MLE for $k>3$ has not been investigated before. Below, we answer this question proving the following theorem:
\begin{Thm}
\label{thm-unique}
Let $k\ge 2$ be a positive integer and $X_{1},\dots,X_{n}$ be i.i.d. random variables with $k$-monotone density $f_{0}\in\mathcal{M}_{k}$. With probability $1$ the MLE $\hat f_{n}$ is uniquely defined by some discrete mixing probability measure $\hat G_{n}$. The cardinality of the support set $\supp(\hat G_{n})$ is less or equal than $n$. 
If $\supp(\hat G_{n})$ consists of points $Y_{1}<\dots<Y_{m}$, $m<n$ then there exists a subset of order statistics $X_{(i_{1})}< \dots <X_{(i_{m})}$ such that:
$$
X_{(i_{j})} < Y_{j} < X_{(i_{j+k})},
$$
where we assume $X_{(i_{l})}=+\infty$ for $l>m$. Also we have $Y_{m}>X_{(n)}$.
\end{Thm}

\section{Proof}
With probability $1$ we have that the order statistics are distinct:
$$
0<X_{(1)}<\dots < X_{(n)}.
$$
We assume that $X_{(l)}=+\infty$ for $l>n$.

The proof consists of several lemmas. First lemma provides a tool for counting zeroes of a function.
\begin{Lem}
\label{app-gen-Rolle}
Let $A\subseteq\Real$ be an interval. Consider a function $f\in C^{k}(A)$. Then the following inequality is true for the number of zeroes $N(f,A)$ of $f$ counted with multiplicities:
$$
N(f^{(k)}, A)\ge N(f,A) - k.
$$
Here the multiplicity of zero $x$ of the function $f\in C^{k}(A)$, $k\ge 0$, $x\in A$ is an integer $n\le k+1$ such that:
$$
f(y) = \dots = f^{(n-1)}(y)=0
$$
and either $n=k+1$ or $f^{(n)}(y)\neq 0$.
\end{Lem}
\begin{proof}
Let $x_{i}$, $1\le i \le m$ be zeroes of $f$ on $A$ and $n_{i}$ be its multiplicities. Then by Rolle's theorem the derivative $f'\in C^{k-1}(A)$ has $m-1$ zeroes between points $x_{i}$ with multiplicities at least $1$ and also zeroes $x_{i}$ with multiplicities $n_{i}-1$. Thus:
$$
N(f',A)\ge m-1 + \sum_{i=1}^{m} (n_{i}-1) = -1 + \sum_{i=1}^{m} n_{i} = N(f,A) - 1.
$$
Applying this inequality $k$ times we obtain the result.
\end{proof}

Proof of the second lemma builds the necessary tools for the main result. Note, that discreteness of the mixing measure $\hat G_{n}$ actually follows from general results about the MLE for mixture models proved in \cite{MR684866}. We repeat a part of that argument to make our proof self-contained and also to provide references to the classic convex analysis textbook \cite{MR0274683} for an interested reader.
\begin{Lem}
\label{lem-unique-supp}
There exists a set $Y$ of points $Y_{1}<\dots <Y_{m}$, $m\le n$ such that $k-$monotone density $\hat f_{n}$ defined by a mixing probability measure $\hat G_{n}$ is the MLE if and only if $\supp(\hat G_{n})\subseteq Y$. In other words, any MLE
has the form:
\begin{eqnarray}
\label{mle-repres}
\hat f_{n}(x) = \sum_{j=1}^{m} a_{j}\frac{k(Y_{j}-x)_{+}^{k-1}}{Y_{j}^{k}},
\end{eqnarray}
where $a_{j}\ge 0$, $\sum a_{j}=1$ and $Y_{j}> X_{(j)}$. Moreover, the vector $(\hat f_{n}(X_{(i)}))_{i=1}^{n}$ is unique. 
\end{Lem}
\begin{proof} Following \cite{MR684866,Lindsay:95} we define a curve $\Gamma$ parameterized by $y\in\RealP$ as:
$$
\Gamma(y) = \left(\frac{1}{y}\ g_{0}\left(\frac{x}{y}\right)\right)_{i=1}^{n} = \left(\frac{k(y-X_{(i)})_{+}^{k-1}}{y^{k}}\right)_{i=1}^{n}.
$$
Since $g_{0}(x/y)$ as a function of $y$ is bounded and equals zero at $0$ and $+\infty$ we have that $\Gamma$ is compact. Suppose that $f\in\mathcal{M}_{k}$ with corresponding mixing probability measure $G$. Then the vector $(f(X_{(i)}))_{i=1}^{n}$ belongs to the set $\bar\Gamma = \cl(\conv(\Gamma))$:
\begin{eqnarray*}
(f(X_{(i)}))_{i=1}^{n} = \int_{0}^{+\infty}\Gamma(y)dG(y)
\end{eqnarray*}
Since $\Gamma$ is compact it follows that $\bar\Gamma$ is also compact and $\bar\Gamma = \conv(\Gamma)$ by (\cite{MR0274683}, Theorem 17.2). Therefore, the continuous function $\prod_{i=1}^{n}z_{i}$ attains its maximum on $\bar\Gamma$, where $z_{i}$ are the coordinates in the space $\RealP^{n}$. Let us denote $S\equiv \argmax_{\bar\Gamma} \prod_{i=1}^{n}z_{i}$.
 
Since the intersection of $\Gamma$ and the interior $\ri(\RealP^{n})$ of $\RealP^{n}$ is not empty we have $S\subset \ri(\RealP^{n})$. The function $\prod_{i=1}^{n}z_{i}$ is strictly concave, therefore $S$ consists of  a single point $b=(b_{i})_{i=1}^{n}>0$. Therefore for any MLE $\hat f_{n}$ it follows that the vector $(\hat f_{n}(X_{(i)}))_{i=1}^{n}$ is equal to $b$ and unique. The gradient of $\prod_{i=1}^{n}z_{i}$ at $b$ is proportional to $1/b\equiv(1/b_{i})_{i=1}^{n}$. 

We have $\dim(\bar \Gamma) = n$. Indeed if we consider $n$ points $t_{i}\in (X_{(i)}, X_{(i+1)})$ then the vectors $\Gamma(t_{i})$ are linearly independent (\cite{MR0053177} or direct observation).
By (\cite{MR0274683}, Theorem 27.4) the vector $1/b$ belongs to the normal cone of $\bar\Gamma$ at $b$. Since $1/b>0$ we have  $b\in \partial \bar\Gamma$ and the plane $\alpha$ defined  by the equation $\sum_{i=1}^{n} z_{i}/b_{i} = n$ is a support plane of $\bar\Gamma$ at $b$. Thus for $v_{i} = k/nb_{i}$ we have:
$$
p(y) \equiv y^{k} - \sum_{i=1}^{n} v_{i} (y-X_{(i)})_{+}^{k-1}\ge 0,
$$
for all $y\ge 0$ and $p(y)=0$ if and only if $y=0$ or  $\Gamma(y)\in\alpha$. Let us denote by $Y$ the set of $y$ such that $\Gamma(y)\in\alpha$ i.e. $\Gamma(Y)=\alpha\cap\Gamma$.

The intersection $\alpha\cap\bar\Gamma$ is an {\it exposed face} (\cite{MR0274683}, p.162) of $\bar\Gamma$. By Theorem 18.3 from \cite{MR0274683} we have $\alpha\cap\bar\Gamma = \conv(\Gamma(Y))$ and by  Theorem 18.1 we have $\supp(\hat G_{n})\subseteq Y$.

The function $p(y)$ belongs to $C^{(k-2)}(\RealP)$; at $y=0$ it has a zero of multiplicity $k-1$ and at each $y\in Y$ it has a zero of multiplicity greater or equal than $2$. By Lemma~\ref{app-gen-Rolle} we have:
$$
N(p^{(k-2)},[0,X_{(j)}]) \ge N(p,[0,X_{(j)}]) - (k-2) \ge 2\#(Y\cap [0,X_{(j)}]) + 1.
$$
The function $p^{(k-2)}(y)$ is piecewise quadratic and therefore on each interval $[X_{(i)},X_{(i+1)}]$ it may have at most two zeroes and on $[0,X_{(1)})$ it has only one zero at zero. Thus:
$$
1+2(j-1) \ge N(p^{(k-2)},[0,X_{(j)}]) \ge 2\#(Y\cap [0,X_{(j)}]) + 1
$$
or equivalently the set $Y$ has no more than $j-1$ points on $[0,X_{(j)}]$. In particular, the set $Y$ contains no more than $n$ points. Let $Y_{1}<\dots<Y_{m}$ be the points of $Y$. Then we have $Y_{j}>X_{(j)}$. 

This implies that for any MLE $\hat f_{n}$ the support of the corresponding mixing measure $\hat G_{n}$ is a subset of $Y$ and thus any MLE $\hat f_{n}$ has the form:
\begin{eqnarray*}
\hat f_{n}(x) = \sum_{j=1}^{m} a_{j}\frac{k(Y_{j}-x)_{+}^{k-1}}{Y_{j}^{k}},
\end{eqnarray*}
where $a_{j}\ge 0$, $\sum a_{j}=1$. 
\end{proof}
The last lemma proves uniqueness of the MLE.
\begin{Lem}
\label{lem-unique-coeff}
The discrete mixing probability measure $\hat G_{n}$ which defines an MLE is unique. 
\end{Lem}
\begin{proof}
Suppose there exist two different MLEs $\hat f^{1}_{n}$ and $\hat f^{2}_{n}$. Then by Lemma~\ref{lem-unique-supp} we have:
$$
\hat f^{l}_{n}(x) = \sum_{j=1}^{m} a^{l}_{j}\frac{k(Y_{j}-x)_{+}^{k-1}}{Y_{j}^{k}}
$$
and therefore the function $q(x)$ defined by:
$$
q(x) = \hat f^{1}_{n}(x) - \hat f^{2}_{n}(x) = \sum_{j=1}^{m} s_{j}\frac{k(Y_{j}-x)_{+}^{k-1}}{Y_{j}^{k}},
$$
where $s_{j} = a^{1}_{j}-a^{2}_{j}$, admits at least $n$ zeroes $X_{(i)}$ and not all $s_{j}$ are equal to zero. Therefore, by Lemma~\ref{app-gen-Rolle}:
$$
N(q^{(k-2)},[0,X_{i}))\ge N(q,[0,X_{i})) - (k-2) \ge i-1 - (k-2).
$$
The function $q^{(k-2)}$ is piecewise linear with knots at $Y_{j}$. Therefore, $N(q^{(k-2)},[0,X_{i})) \le \#\{Y\cap[0, X_{(i)})\} + 1$ and:
$$
\#\{Y\cap[0, X_{(i)})\} \ge i - k.
$$
This implies:
$$
Y_{i-k} < X_{(i)},
$$
and by Lemma~\ref{lem-unique-supp}
$$
Y_{i-k} < X_{(i)} < Y_{i}.
$$
Now, by Corollary 1 of \cite{MR0053177} we have $\det((Y_{i}-X_{(j)})_{+}^{k})_{i,j=1}^{m}>0$. Thus the vectors $\Gamma(Y_{i})$, $1\le i\le m$ are independent and all $s_{i}=0$. This contradiction completes the proof.
\end{proof}

Note, that from Lemma~\ref{lem-unique-supp} it follows that any mixing probability measure on a finite set $Y=\{Y_{j}\}$ defines MLE. 
Lemma~\ref{lem-unique-coeff} shows that the MLE is unique and hence the vectors $\Gamma(Y_{j})$ are independent. 
In particular, there exists a subset of order statistics $X_{(i_{1})}<\dots <X_{(i_{m})}$ such that:
$$
\det\left((Y_{j}-X_{(i_{l})})_{+}^{k-1}\right)_{j,l=1}^{m}>0.
$$
Therefore, by Corollary 1 of \cite{MR0053177} we obtain $Y_{j-k}<X_{(i_{j})}<Y_{j}$.

The equality~\eqref{mle-repres} also implies that for any MLE $Y_{m}>X_{(n)}$ since otherwise the last coordinate of all vectors $\Gamma(Y_{j})$ is equal to zero and $\hat f_{n}(X_{(n)})=0$
 and this completes the proof of Theorem~\ref{thm-unique}.

\nocite{*}

\section*{Acknowledgments}
I would like to thank my advisor, Professor Jon A. Wellner, for introducing this problem to me and for helpful comments and suggestions.

\bibliographystyle{ims}
\bibliography{Uniqueness-of-k-monotone-MLE}

\end{document}